\documentclass[12pt]{amsart}
\usepackage{txfonts}

\usepackage{amssymb}
\usepackage{mathrsfs}
\usepackage{amsmath}
\usepackage{amssymb,latexsym}

\def\C{\mathbb C}

\def\R{\mathbb R}

\def\Q{\mathbb Q}
\def\Z{\mathbb Z} 
\def\R{\mathbb R} 
\def\C{\mathbb C} 

\newtheorem{thm}{Theorem}[section]
\newtheorem{cor}[thm]{Corollary}

\newtheorem{prop}[thm]{Proposition}

\theoremstyle{definition}
\newtheorem{defn}[thm]{Definition}

\numberwithin{equation}{section}

\begin{document}
\baselineskip=17pt

\title [The ratio of $\theta$-congruent numbers]{The ratio of $\theta$-congruent numbers}
\author[Yan Li]{Yan Li}
\email{liyan\_00@mails.tsinghua.edu.cn}
\author[Su Hu]{Su Hu}
\address{Department of Mathematical Sciences\\ Tsinghua University\\
Beijing 100084, China (Su Hu)} \email{hus04@mails.tsinghua.edu.cn}

\begin{abstract}Let $0<\theta<\pi$ such that $\cos\theta\in \Q$. In this paper, we prove that for given positive
square-free coprime integers $k,l$, there exist infinitely many
pairs $(M,N)$ of $\theta$-congruent numbers such that $lN=kM$. This
generalize the previous result of Rajan and Ramaroson~\cite{RajanR}
on the ratio of congruent numbers from congruent numbers (i.e.
$\theta=\pi/2$) to arbitrary $\theta$-congruent numbers.
\end{abstract}

\subjclass[2000]{Primary 11G05;
Secondary 14G05} \maketitle

\section{Introduction}
A congruent number is a square-free integer which is the area of a
right triangle with rational sides. The interesting problem to
decide integers which are congruent numbers was systematically
studied by Arab scholars in the tenth century. The congruent number
problem are closely related to the arithmetic theory of elliptic
curves. It is well known that a square-free integer $n$ is a
congruent number if and only if the elliptic curve $ny^2=x(x^2-1)$
has positive rank (see~\cite{Koblitz}, for instance).

Fujiwara~\cite{Fujiwara} extended the concept of congruent
numbers by considering general (not necessarily right) triangles
with rational sides. Let $\theta$ be a real number with
$0<\theta<\pi$. A triangle with an angle $\theta$ and rational sides
is called a rational $\theta$-triangle. Notice that, for such a
triangle, $\cos\theta$ is necessarily rational. In the sequel, we
always assume $\cos\theta\in \mathbb{Q}$ and denote
\begin{equation}\label{eq:1}
\cos\theta=\frac{s}{r},\ r,s\in \mathbb{Z},\ r>0,\ {\mathrm
gcd}(s,r)=1.
\end{equation}
\begin{defn}A natural number $n$ is $\theta$-congruent if
$n$ is square-free and $nr\sin\theta$ is the area of a rational
$\theta$-triangle.
\end{defn}
For $\theta=\frac{\pi}{2}$, $\theta$-congruent numbers are just the
usual congruent numbers. Like the congruent numbers, the
$\theta$-congruent numbers are also closely related to the
arithmetic of elliptic curves. Let $E_{n,\theta}$ be the elliptic
curve defined by
$$E_{n,\theta}:\ ny^2=\frac{1}{r}x(x+\cos \theta-1)(x+\cos
\theta+1).
$$
In a slight different form, Fujiwara~\cite{Fujiwara} showed that
\begin{thm}(Fujiwara)\label{Fujiwara} Let n be any square-free
natural number. Then

(1) $n$ is $\theta$-congruent if and only if $E_{n,\theta}$ has a
rational point of order greater than $2$.

(2) For $n\neq1,2,3,6$, $n$ is $\theta$-congruent if and only if
$E_{n,\theta}(\mathbb Q)$ has positive rank.
\end{thm}
For more properties of $\theta$-congruent numbers,
see~\cite{Goto},~\cite{Kan}.

For congruent numbers, Chahal~\cite{Chahal} has proved that there
exist infinitely many congruent numbers in each residue class modulo
8. Bennett~\cite{Bennett} extended Chahal's result to any integer
$m>1$. Johnstone and Spearman~\cite{JhonstoneS} made further
improvements on Bennett's result.

Recently, Rajan and Ramaroson ~\cite{RajanR}  got the following
interesting result on the ratio of congruent numbers.
\begin{thm}(Rajan and Ramaroson) If $k$ and $l$ are positive, square-free coprime
integers, then there exist infinitely many pairs $(M,N)$ of
congruent numbers such that $lN=kM$.
\end{thm}

In this paper, following the method of Rajan and Ramaroson
~\cite{RajanR}, we will generalize their results to arbitrary
$\theta$-congruent numbers. Our main results are the following:

\begin{thm}Let $0<\theta<\pi$ such that $\cos\theta\in \Q$. If $k$ and $l$ are
positive, square-free coprime integers, then there exist infinitely
many pairs $(M,N)$ of $\theta$-congruent numbers such that $lN=kM$.
\end{thm}

\begin{cor}Assumption as above, there exist infinitely many
square-free integers $N$ such that both $kN$ and $lN$ are
$\theta$-congruent numbers.
\end{cor}
\section{Generalized Holm's Curve and its jacobian}
Let $\beta$ be a rational number such that $\beta\neq\pm1$. Let
$k,l$ be coprime positive integers such that $k\neq l$. We call the
curve
$$H_{\beta}:\ lx(x+\beta-1)(x+\beta+1)=ky(y+\beta-1)(y+\beta+1)
$$
the generalized Holm's curve since for $\beta=0$, $H_{\beta}$ was
considered by Holm~\cite{Holm} in a slight different form.

   In this section, we will study $H_{\beta}$ and its jacobian
   $E_{\beta}$. We will show that $H_{\beta}$ is a smooth
   irreducible curve of genus one with infinitely many rational
   points.
\begin{prop}$H_{\beta}$ is a smooth
   irreducible curve of genus one.
\end{prop}
\begin{proof}
It is well-known that a smooth cubic is automatically irreducible of
genus one. So we only need to check the smoothness. To do this, we
will use the jacobian criterion.

Let $G(x,y)=lx(x+\beta-1)(x+\beta+1)-ky(y+\beta-1)(y+\beta+1)$. The
equation
\begin{equation}\label{eq:2}
 \left\{
   \begin{aligned}
     \frac{\partial G}{\partial
    x}&=l(x(x+\beta-1)+x(x+\beta+1)+(x+\beta-1)(x+\beta+1))=0\\
  \frac{\partial G}{\partial y}&=-k(y(y+\beta-1)+y(y+\beta+1)+(y+\beta-1)(y+\beta+1))=0  \\
   \end{aligned}
  \right.
  \end{equation}
has four solutions
$(\alpha_1,\alpha_1),(\alpha_1,\alpha_2),(\alpha_2,\alpha_1),(\alpha_2,\alpha_2),$
where we assume $\alpha_1<\alpha_2$. Since $\alpha_1$ and $\alpha_2$
are the extreme points of cubic function
$u=v(v+\beta-1)(v+\beta+1)$, from the graph, one can see
$\alpha_1(\alpha_1+\beta-1)(\alpha_1+\beta+1)>0$ and
$\alpha_2(\alpha_2+\beta-1)(\alpha_2+\beta+1)<0$. Since $k\neq l$
are positive integers, the four points
$(\alpha_1,\alpha_1),(\alpha_1,\alpha_2),(\alpha_2,\alpha_1),(\alpha_2,\alpha_2)$
do not satisfy the equation $G(x,y)=0$. Thus $H_{\beta}$ is smooth
on the affine part. It is easily checked that $H_{\beta}$ is also
smooth at the infinite points
$(\sqrt[3]{k/l},1,0),(\sqrt[3]{k/l}\rho,1,0),(\sqrt[3]{k/l}\rho^2,1,0),$
where $\rho$ is a primitive cubic root of unity.
\end{proof}

Using the method of~\cite{ST} (p.23), one can change $H_{\beta}$ to
its jacobian $E_{\beta}$: $$Y^2=X^3-
 \frac{1}{3} k^2 l^2 (3 + \beta^2)^2 X+\frac{1}{27} k^2 l^2 (2k l
\beta^2 (-9 + \beta^2)^2 +
    27 k^2 (-1 + \beta^2)^2 + 27 l^2 (-1 + \beta^2)^2)
$$  by identifying $(0,0)$ to the zero element.
The computation is done by Mathematica 7.0, see the file ``Eq.nb" in
the supplementary materials.

\begin{prop}\label{prop:1}(i) The discriminant of $E_\beta$ is
$$-\frac{16}{27} k^4 l^4 (-4 k^2 l^2 (3 + \beta^2)^6 + (2 k l \beta^2
\ (-9 + \beta^2)^2 + 27 k^2 (-1 + \beta^2)^2 +
     27 l^2 (-1 + \beta^2)^2)^2).$$
     (ii) The $j$-invariant of $E_\beta$ is
     $$-\frac{6912 k^2 l^2 (3 + \beta^2)^6}{-4 k^2 l^2 (3 + \beta^2)^6 + \
(2 k l \beta^2 (-9 + \beta^2)^2 + 27 k^2 (-1 + \beta^2)^2 +
    27 l^2 (-1 + \beta^2)^2)^2}
     $$
     (iii) The rational transformation relating $H_{\beta}$ and
     $E_{\beta}$ are
     \begin{equation*}\begin{aligned}X&= \frac{
 k l (-3 l y +
    k (3 x + 6 \beta - (3 x + 4 y) \beta^2 - 6 \beta^3) +
    l \beta (-6 + 4 x \beta + 3 \beta (y + 2 \beta)))}{
 3 l x - 3 k y},\\
 Y &= -\frac{
    k (k -
    l) l (-1 + \beta^2) (k l (x - y) (1 + 2 (x + y) \beta +
       3 \beta^2) - l^2 x (-1 + \beta (x + \beta)) +
    k^2 y (-1 + \beta (y + \beta)))}{(l x - k y)^2},
     \end{aligned}
     \end{equation*}

\begin{equation*}\begin{aligned}
&\ \ \ \ \ \ \ \ x =\\ -&\frac{3 k (-1 + \beta^2) (l (2 k^2 \beta^2
(-1 + \beta^2) +
          3 l^2 (-1 + \beta^2)^2 +
          k l (3 + 16 \beta^2 - 3 \beta^4)) +
       6 \beta Y -
       3 (k + l - (k - 3 l) \beta^2) X)}{9 (k +
       l) (-1 + \beta^2) Y + \beta (k l (9 k^2 (-1 + \beta^2)^2 +
          9 l^2 (-1 + \beta^2)^2 -
          2 k l (9 + 5 \beta^2 (-6 + \beta^2))) -
       6 k l (3 + 5 \beta^2) X + 18
X^2)},\\ &\ \ \ \ \ \ \ \ y=\\ -&\frac{3 l (-1 + \beta^2) (k (2 l^2
\beta^2 (-1 + \beta^2) +
          3 k^2 (-1 + \beta^2)^2 +
          k l (3 + 16 \beta^2 - 3 \beta^4)) +
       6 \beta Y -
       3 (k + l + (3 k - l) \beta^2) X)}{9 (k +
       l) (-1 + \beta^2) Y + \beta (k l (9 k^2 (-1 + \beta^2)^2 +
          9 l^2 (-1 + \beta^2)^2 -
          2 k l (9 + 5 \beta^2 (-6 + \beta^2))) -
       6 k l (3 + 5 \beta^2) X + 18 X^2)}
.\end{aligned}\end{equation*}

(iv):Under the above transformation, the nine rational points of
$H_\beta$ correspond to the nine rational points of $E_{\beta}$ as
follows:
\begin{equation*}\begin{aligned}
&P_1: (-\beta-1,-\beta+1)\leftrightarrow(\frac{1}{3} k l (3 +
\beta^2),k (k - l) l (-1 + \beta^2)),
\\
&P_2:(0,-\beta+1)\leftrightarrow(\frac{1}{3} l (3 l (1 + \beta)^2 -
2 k \beta (3 + \beta)),(k - l) l (1 + \beta) (k (-1 + \beta) - l (1
+ \beta)^2)), \\
&P_3: (-\beta+1,-\beta+1)\leftrightarrow(\frac{1}{3} k l (-3
+ (-6 + \beta) \beta),k l (k + l) (-1 + \beta^2)),\\
 &P_4: (-\beta-1,0)\leftrightarrow(\frac{1}{3} k (3 k
(-1 + \beta)^2 - 2 l (-3 + \beta) \beta),k (k - l) (-1 + \beta) (l +
k (-1 + \beta)^2 + l \beta)),\\
&P_5:(0,0)\leftrightarrow O,
\\ &P_6:(-\beta+1,0)\leftrightarrow (\frac{1}{3} k (3 k (1 +
\beta)^2 - 2 l \beta (3 + \beta)),k (k - l) (1 + \beta) (l - l \beta
+ k (1 + \beta)^2)),
\\ &P_7: (-\beta-1,-\beta-1)\leftrightarrow(\frac{1}{3} k l (-3 + \beta (6 + \beta)),-k l (k + l) (-1 +
\beta^2)),
\\
&P_8: (0,-\beta-1)\leftrightarrow(\frac{1}{3} l (3 l (-1 + \beta)^2
- 2 k (-3 + \beta) \beta),-(k - l) l (-1 + \beta) (l (-1 + \beta)^2
+ k (1 + \beta))),
\\ &P_9: (-\beta+1,-\beta-1)\leftrightarrow(\frac{1}{3} k l (3 + \beta^2),-k (k - l) l (-1 +
\beta^2)).
\end{aligned}\end{equation*}

(v) $E_\beta(\R)[2]=\Z/2\Z$.
\end{prop}

\begin{proof} The (i), (ii), (iii) (iv) can be checked by Mathematica
7.0, see files ``invariants.nb", ``trans.nb", ``coordinates.nb" in
the supplementary materials. We only give the proof of (v). Let
$E_{\beta}: Y^2=f(X)$. Since
$$f'(\pm kl(3+\beta^2)/3)=0
$$
and
$$f(\frac{1}{3}kl(3+\beta^2))=k^2 (k - l)^2 l^2 (-1 + \beta^2)^2>0,
$$
$f(X)=0$ has only one real root. Therefore, $E_\beta(\R)[2]=\Z/2\Z$.
\end{proof}

The expression of addition of two points of $E_{\beta}$ is quite
complicated since it has three parameters $k,l,\beta$. So it is hard
to check a point on $E_{\beta}$ has infinite order by computer.
Instead, we will give the following geometric and intuitive proof.
The proof is based on Mazur's famous result on the torsion group of
rational points of elliptic curves over $\Q$ (see \cite{Mazur1},
~\cite{Mazur2}).
\begin{prop} $E_{\beta}(\Q)$ has positive rank. Equivalently,
$H_{\beta}$ has infinitely many rational points.
\end{prop}
\begin{proof}By (v) of Proposition \ref{prop:1} and Mazur's
result, if the rank of $E_{\beta}(\Q)$ is 0, then
$|E_{\beta}(\Q)|\leq 12$. So in order to prove $E_{\beta}(\Q)$ has
positive rank, we will show it has at least 13 rational points.
Since $E_{\beta}$ and $H_{\beta}$ are birationally equivalent over
$\Q$, it suffices to check it for $H_{\beta}$. Changing the
coordinate $(x+\beta, y+\beta)$ to $(x,y)$, we get a new equation
$$H_{\beta}:l(x-\beta)(x-1)(x+1)=k(y-\beta)(y-1)(y+1).$$

From this point to the end of the proof, we will use the above
equation to show $|H_{\beta}(\Q)|\geq 13$. To do this, we can assume
$0<\beta<1$ and $l>k$. The reason is as follows. We omitted the case
of $\beta=0$ since it was already done in \cite{RajanR}. For
$\beta<0$, change coordinate $(x,y)$ to $(-x,-y)$ and for $\beta>1$,
change coordinate $(x,y)$ to
$(\frac{1}{2}(\beta+1)(x+1)-1,\frac{1}{2}(\beta+1)(y+1)-1)$.

Letting $(x-\beta)(x-1)(x+1)=(y-\beta)(y-1)(y+1)=0$, we get nine
distinct rational points of $H_{\beta}$:
\begin{equation*}\begin{aligned}
&P_1:(-1,-1),\ \ \ &P_2:(-1,\ \beta),\ \ \ \ &P_3:(-1,+1),\\
&P_4:(\ \beta,-1),\ \ &P_5:(\ \beta, \ \ \beta),\ \ \ \ \  &P_6:(\ \beta,+1),\\
&P_7:(+1,-1),\ \ \ &P_8:(+1,\ \beta),\ \ \ \ &P_9:(+1,+1)
.\end{aligned}\end{equation*}

The tangent line of $H_{\beta}$ at $(\beta,\beta)$:
$l(x-\beta)-k(y-\beta)=0$ meets $H_{\beta}$ at the new point:
\begin{equation*}\begin{aligned}
P_0:\ (+\frac{l-k}{l+k}\beta,\ -\frac{l-k}{l+k}\beta).
\end{aligned}\end{equation*}
\begin{picture}(200,200)
\put(0,100){\vector(1,0){200}} \put(100,0){\vector(0,1){200}}
\put(180,180){\circle*{2}} \put(180,185){\makebox(10,10){$P_3$}}
\put(180,140){\circle*{2}} \put(180,155){\makebox(10,10){$P_6$}}
\put(180,20){\circle*{2}} \put(180,25){\makebox(10,10){$P_9$}}
\put(140,180){\circle*{2}} \put(140,185){\makebox(10,10){$P_2$}}
\put(140,140){\circle*{2}} \put(140,155){\makebox(10,10){$P_5$}}
\put(140,20){\circle*{2}} \put(140,25){\makebox(10,10){$P_8$}}
\put(20,180){\circle*{2}} \put(20,185){\makebox(10,10){$P_1$}}
\put(20,140){\circle*{2}} \put(20,155){\makebox(10,10){$P_4$}}
\put(20,20){\circle*{2}} \put(20,25){\makebox(10,10){$P_7$}}
\put(120,80){\circle*{2}} \put(120,85){\makebox(10,10){$P_{0}$}}
\end{picture}
\put(-105,-20){\makebox(20,10)[lb]{\bf{Figure}}}

From the group law of cubic curve (p.18-22 of~\cite{ST}), the
negative $-P$ is the third intersection of cubic curve and the line
passing through $P$ and $P_0$, where $P$ is an arbitrary point of
$H_\beta$. So we have $P_1+P_9=O$. From the figure above, one can
see $P_6+P_7$ might be $O$. By (iv) of Proposition ~\ref{prop:1}, it
is easy to see the $Y$-coordinates of $P_2,P_3,P_4,P_8$ do not equal
to zero under the assumption $0<\beta<1$ and $l>k>0$. So they are
not two torsion points, i.e, $-P_i\neq P_i\ (i=2,3,4,8)$. Drawing
the lines connecting $P_i\ (i=2,3,4,8)$ and $P_0$, one can easily
see the set $\{-P_2, -P_3,-P_4, -P_8\}$ is disjoint with the set
$\{P_i|1\leq i\leq 9\}$ from the figure above. Hence
$|H_{\beta}(\Q)|\geq 13$, which concludes the proof.
\end{proof}

 In what following, we will apply the above results to the ratio of
 $\theta$-congruent numbers. From this point on, we will fixed an angle
 $\theta$ with $0<\theta<\pi$ and let $\beta=\cos\theta$. We will
 use notations $H_{\theta}$, $E_\theta$ instead of $H_{\beta}$,
 $E_\beta$, respectively. Also we will fixed coprime, unequal positive integers $k, l$. In addition,
 we require $k, l$ are square-free. Recall that $r$ is the denominator of $\cos\theta$.

 Let
$A_x=x((x+\beta)^2-1)/r$ and
 $A_y=y((y+\beta)^2-1)/r$. Then every rational point $(x,y)$ on $H_{\theta}$
 with $A_x>0$
 gives rise two rational $\theta$-triangles
 whose areas are in the ratio
 $$\frac{A_x }{A_y}=\frac{k}{l}.
 $$
 Indeed, if $A_x$ is positive, the rational
 $\theta$-triangle $$\{(x+\beta)^2-1, 2x, 1+(x+\beta)^2-2(x+\beta)\beta\}\
 \ (\mathrm{for}\
 \ x>0)$$ or $$\{1-(x+\beta)^2, -2x, 1+(x+\beta)^2-2(x+\beta)\beta\}\ \
 (\mathrm{for}\ \
 x<0)$$ has area $A_xr\sin\theta$, and similar for $A_y$ since it is also positive. Therefore, by the definition of
$\theta$-congruent numbers, every rational point $(x,y)$ on
$H_{\theta}$ with $A_x>0$ produces a pair of congruent numbers,
$(N_x,N_y)$ when we take the square-free parts $N_x$ of $A_x$ and
$N_y$ of $A_y$ respectively.

In the sequel, demonstrating the transformations in (iii) of
proposition~\ref{prop:1}, we will show that there are infinitely
many points $(X,Y)$ on $E_{\theta}(\mathbb{\Q})$ such that the
corresponding point $(x,y)$ on $H_{\theta}$ satisfies $$A_x>0,\
A_y>0,\ (l,N_x)=1,\ (k,N_y)=1.$$ Taking the square-free part of both
sides of $lA_x=kA_y$, we will get
\begin{equation}\label{eq:5}lN_x=kN_y.
\end{equation} To show the infinity of such points on
$E_{\theta}(\mathbb{\Q})$, we use the valuation properties of
elliptic curves over local fields.

\section{The valuation properties of global points on elliptic curves}
Let $E$ be an elliptic curve over $\Q$ defined by the Weierstrass
equation: $$Y^2=X^3+aX+b,\ a,b\in\Q.$$ Let $S=\{p_1,p_2,...,p_t\}$
be a set of prime numbers.

Given positive integers $m_1,m_2,...,m_t$, let
$$U_{m_1,m_2,...,m_t}(E)=\{P\in E(\Q)|{\rm ord}_{p_i}(X(P))=-2m_i,\ {\rm ord}_{p_i}(Y(P))=-3m_i,\ {\rm where}\
1\leq i\leq t\}.
$$
Improving the Proposition 3.3 of~\cite{RajanR} a bit more, we get
\begin{prop}\label{prop:2}If $E(\Q)$ has positive rank, then there exists an integer
$N$ such that $U_{m_1,m_2,...,m_t}(E)\neq \varnothing$, for all
$m_1,m_2,...,m_t\geq N$.
\end{prop}
\begin{proof}Let $E'$ with coordinates $(X',Y')$ be the global minimal Weierstrass Equation
of $E$ over $\Q$. It is well-known that such equation exists
(see~\cite{Silverman} and~\cite{Knapp}, for instance). The
coordinates $(X,Y)$ and $(X',Y')$ are related by
$$X=u^2X'+f\ {\rm and}\ Y=u^3Y'+gu^2X'+h,\ u,f,g,h\in\Q,\ u\neq0.
$$
Hence, there exists an integer $M_1$ such that
$$U_{m_1-s_1,m_2-s_2,...,m_t-s_t}(E)=U_{m_1,m_2,...,m_t}(E').
$$for all $m_1,m_2,...,m_t\geq M_1$, where $s_1={\rm ord}_{p_1}(u),...,s_t={\rm ord}_{p_t}(u)$.

By Proposition 3.1 and 3.3 of~\cite{RajanR}, there exists an integer
$M_2$ such that $$U_{m_1,m_2,...,m_t}(E')\neq \varnothing,\ {\rm
for\ all\ } m_1,m_2,...,m_t\geq M_2.
$$
Let $M=\max\{M_1,M_2\}$. Let $N=\max\{M-s_1,M-s_2,...,M-s_t\}$. Then
$$U_{m_1,m_2,...,m_t}(E)\neq \varnothing,\ {\rm
for\ all\ } m_1,m_2,...,m_t\geq N.
$$
\end{proof}
\begin{prop}\label{prop:3}Assume $E(\R)[2]=\Z/2\Z$ and $E(\Q)$ has positive
rank. Let $Q_1,Q_2\in E(\Q)$ such that $Q_1$ has infinite order.
Then for any $ M>0$, there are infinitely many points $P$ belonging
to the set $\{\pm([n]Q_1+Q_2)|n\in\Z\}$ such that $X(P)>M$ and
$y(P)>0$.
\end{prop}
\begin{proof}Since $E(\R)[2]=\Z/2\Z$, $E(\R)$ is a 1-dimensional
commutative connected compact lie group. Hence, $E(\R)$ is
isomorphic to the unit circle group $\{z\in\C|\ |z|=1\}$ as lie
groups (see p.7 of \cite{Poonen} or p.42 of \cite{ST}, for
instance). The image of the set $\{[n]Q_1+Q_2|n\in\Z\}$ are
everywhere dense in the circle group since $Q_1$ has infinite order.
Pick up a sequence $z_n$ from the image set such that $\lim z_n=1$
as $n$ goes to infinity. The corresponding sequence $R_n$ on $E(\R)$
goes to the zero element as $n$ goes to infinity. This is equivalent
to $\lim X(R_n)=+\infty$. Taking negative if necessary, we can
assume $Y(R_n)>0$ for all $n$. This concludes the proof.
\end{proof}
\section{Proof of the main theorem} We now apply Proposition~\ref{prop:2} and~\ref{prop:3} to the curve $E_{\theta}$. From the results in section 2, we know
that $E_{\theta}(\Q)$ has positive rank and
$E_{\theta}(\R)[2]=\Z/2\Z$. Recall that $k\neq l$ are coprime
square-free positive integers. So $E_{\theta}$ satisfies the
conditions of Proposition~\ref{prop:2} and~\ref{prop:3}. Let $S$ be
the set of prime divisors of $k$ and $l$.

In what following, we will find an infinite set $\mathcal{P}$ such
that every point $(X,Y)$ of $\mathcal{P}$ satisfies
equation~\ref{eq:5}.

By Proposition~\ref{prop:2},
\begin{equation}\label{eq:3}U_{m_1,m_2,...,m_t}(E_{\theta})\neq \varnothing \end{equation}
for $m_1,m_2,...,m_t$ being sufficiently large.

Let $P=(X,Y)$ be an arbitrary element of
$U_{m_1,m_2,...,m_t}(E_{\theta})$. Look carefully at the
transformation formula in $(iii)$ of Proposition~\ref{prop:1}. For
$1\leq i\leq t$, we have
\begin{equation}\label{eq:4}\begin{aligned}
&{\rm ord}_{p_i}(x)={\rm ord}_{p_i}(6\beta Y)-{\rm
ord}_{p_i}(18\beta X^2)={\rm ord}_{p_i}(3)-m_i,\\&{\rm
ord}_{p_i}(y)={\rm ord}_{p_i}(6\beta Y)-{\rm ord}_{p_i}(18\beta
X^2)={\rm ord}_{p_i}(3)-m_i,\\&{\rm ord}_{p_i}(A_x)=3{\rm
ord}_{p_i}(x)-{\rm ord}_{p_i}(r)=3{\rm ord}_{p_i}(3)-3m_i-{\rm
ord}_{p_i}(r),\\&{\rm ord}_{p_i}(A_y)=3{\rm ord}_{p_i}(y)-{\rm
ord}_{p_i}(r)=3{\rm ord}_{p_i}(3)-3m_i-{\rm ord}_{p_i}(r),
\end{aligned}
\end{equation}
if $m_1,m_2,...,m_t$ are sufficiently large.

Therefore, we can fix suitable, large $m_1,m_2,...,m_t$ such that
equations~\ref{eq:3} and~\ref{eq:4} hold and ${\rm
ord}_{p_i}(A_x),{\rm ord}_{p_i}(A_y)$ are both even for all points
$P=(X,Y)\in U_{m_1,m_2,...,m_t}(E_{\theta})$, where $1\leq i\leq t$.

From the transformation formula in $(iii)$ of
Proposition~\ref{prop:1} again, we have
$$x\approx-\frac{6\beta Y}{18\beta X^2}\approx-\frac{1}{3 \sqrt{X}}\rightarrow0^-\ {\rm and}\ y\approx-\frac{6\beta Y}{18\beta X^2}\approx-\frac{1}{3 \sqrt{X}}\rightarrow0^-
$$
as $X$ goes to $+\infty$ ($Y>0$). So there exists $M>0$ such that if
$X>M$ and $Y>0$, then $-1-\beta<x,\ y<0$ which implies $A_x,\
A_y>0$.

Fix an element $Q\in U_{m_1,m_2,...,m_t}(E_{\theta})$. Let $h$ be a
positive integer coprime to $p_1p_2...p_t$. Set
$Q_1=[p_1p_2...p_t]Q$ and $Q_2=[h]Q$. Applying
Proposition~\ref{prop:3} to such $Q_1,Q_2,M$, we find an infinite
set $\mathcal{P}$ such that
$\mathcal{P}\subset\{\pm([n]Q_1+Q_2)|n\in\Z\}$ and every element
$(X,Y)$ of $\mathcal{P}$ satisfies $X>M$ and $Y>0$. So every element
$(X,Y)$ of $\mathcal{P}$ satisfies $A_x>0,A_y>0$.

Since $[n]Q_1+Q_2=[np_1p_2...p_t+h]Q$ and $(h,\ p_1p_2...p_t)=1$, by
Proposition 3.1 of~\cite{RajanR}, we have
$\mathcal{P}\subset\{\pm([n]Q_1+Q_2)|n\in\Z\}\subset
U_{m_1,m_2,...,m_t}(E_{\theta})$. So every element $(X,Y)$ of
$\mathcal{P}$ satisfies that ${\rm ord}_{p}(A_x),{\rm ord}_{p}(A_y)$
are both even for all $p\in S$.

Recall that $N_x$ (resp. $N_y$) is the square-free part of $A_x$
(resp. $A_y$). Summing up, we have
$$A_x>0,\ A_y>0,\
(l,N_x)=1,\ (k,N_y)=1$$ for every point $(X,Y)$ of $\mathcal{P}$.
Therefore
\begin{thm}\label{th:1}For each $(X,Y)\in \mathcal{P}$, we have $lN_x=kN_y.
$
\end{thm}
Following ~\cite{RajanR}, we get
\begin{thm}\label{th:2} Associated with the infinite set of points $(X,Y)$ in
$\mathcal{P}$, there are infinitely many pairs of square-free
integers $(N_x,N_y)$
\end{thm}
\begin{proof}
Assume that there are only finitely many such pairs. Then there must
exist a pair $(N,M)$ of square-free integers associated with
infinitely many rational points $(X,Y)$ in $\mathcal{P}$. Using
$(x,y)$ instead of $(X,Y)$, we conclude that the curve
$\mathcal{C}:$
\begin{equation*}
 \left\{\begin{aligned}
 &lx(x+\beta-1)(x+\beta+1)=ky(y+\beta-1)(y+\beta+1)\\
&x(x+\beta-1)(x+\beta+1)/r=Nz^2
 \end{aligned}
 \right.
  \end{equation*}
  has infinitely many rational points. Consider the rational map of curves:
  $$\mathcal{C}\rightarrow E_{\beta},\ \ (x,y,z)\mapsto(x,y).
  $$
This map is of degree $2$ and ramified at the point $(x,y)=(0,0)$.
The Riemann-Hurwitz formula implies that the genus of $\mathcal{C}$
is greater than $1$. By Faltings~\cite{Faltings}'s theorem,
$\mathcal{C}$ only has a finite number of rational points. Thus, we
get a contradiction. So there are infinitely many such pairs.
\end{proof}
Putting Theorem~\ref{th:1} and \ref{th:2} together, we get our main
theorem:
\begin{thm}Let $0<\theta<\pi$. If $k$ and $l$ are
positive, square-free coprime integers, then there exist infinitely
many pairs $(M,N)$ of $\theta$-congruent numbers such that $lN=kM$.
\end{thm}

Letting $l=1$, we get the following corollary.
\begin{cor}Let $0<\theta<\pi$. Given a positive, square-free integer
$k$, there exist infinitely many pairs $(M,N)$ of $\theta$-congruent
numbers such that $N=kM$.
\end{cor}

\textbf{Acknowledgement:} The authors would like to thank Professor Sunghan Bae
and Professor Linsheng Yin for their encouragements. This work was partially supported by Postdoctoral Science Foundation of China.

\textbf{Supplementary Material:} The files``Eq.nb",
``invariants.nb", ``trans.nb", ``coordinates.nb" will appear on
line.


\begin{thebibliography}{100000}
\bibitem{Holm} A.Holm, Some points on diophantine analysis, Proc.
Edinburgh Math. Soc. 22 (1903), 40-48.
\bibitem{Mazur1} B. Mazur, Modular curves and the Eisenstein ideal.
IHES Publ. Math. 47 (1977), 33-186.
\bibitem{Mazur2} B. Mazur, Rational isogenies of prime degree.
Invent. Math. 44 (1978), 129-162.


\bibitem{Faltings} G.Faltings, Endlichkeitss$\ddot{a}$tze f$\ddot{u}$r abelsche Variet$\ddot{a}$ten $\ddot{u}$ber
Zahlk$\ddot{o}$rpern, Invent. Math. 73 (1983), no. 3, 349-366.

\bibitem{Chahal}J. S. Chahal, On an identity of Desboves, Proc.
Japan. Acad. Series A 60 (1984), 105-108.

\bibitem{Goto}T. Goto, Calculation of Selmer groups of elliptic curves
with rational 2-torsions and $\theta$-congruent number problem.

\bibitem{Koblitz} N.Koblitz, Introduction to Elliptic Curves and
Modular Forms, Grad. Texts in Math. 97, Springer, New York, 1984.


\bibitem{Silverman} J. H. Silverman, The arithemetic of Elliptic
Curves, Grad. Texts in Math. 106, Springer, NewYork, 1986.

\bibitem{Knapp} A. W. Knapp, Elliptic Curves, Math. Notes 40,
Princeton  Univ. Press, Princeton, NJ, 1992.

\bibitem{ST} J. H. Silverman and J. Tate, Rational points on elliptic curves. Undergraduate Texts in Mathematics.
Springer-Verlag, New York, 1992.
 \bibitem{Fujiwara} M. Fujiwara, $\theta$-congruent numbers.
Number theory (Eger, 1996), 235--241, de Gruyter, Berlin, 1998.

\bibitem{Kan} M. Kan, $\theta$-congruent numbers and ellptic curves,
 Acta Arith. 94 (2000), no. 2, 153-160.
\bibitem{Bennett}M. A. Bennett, Lucas square pyramid problem
revisited, Acta Arith. 105 (2002), no.4, 341-347.
\bibitem{RajanR}  A. Rajan and F. Ramaroson, Ratios of congruent numbers, Acta Arithemetica 128 (2007), no. 2, 101-106.
\bibitem{Poonen} B. Poonen, ¡°Elliptic curves¡±, pp. 183¨C207 in Surveys in
algorithmic number theory, edited by J. P. Buhler and P.
Stevenhagen, Math. Sci. Res. Inst. Publ. 44, Cambridge University
Press, New York, 2008.
\bibitem{JhonstoneS}J. A. Jhonstone and B. K. Spearman, On the distribution of congruent
numbers, Proc. Japan. Acad. Series A 86 (2010), 89-90.
\end{thebibliography}
\end{document}